\date{}

\documentclass[11pt]{article}

\usepackage{amsmath,amsfonts,amssymb,amsthm}
\usepackage{latexsym, xspace, enumerate}
\usepackage[mathscr]{eucal}
\usepackage[all]{xypic}
\usepackage{hyperref}
\usepackage{amscd}
\usepackage{vmargin}

\setmarginsrb{28mm}{10mm}{28mm}{25mm}%
          {10mm}{7mm}{10mm}{10mm}

\newtheorem{theorem}{Theorem}[section]
\newtheorem{problem}[theorem]{Problem}

\newtheorem{proposition}[theorem]{Proposition}
\newtheorem{lemma}[theorem]{Lemma}
\newtheorem{fact}[theorem]{Fact}
\newtheorem{corollary}[theorem]{Corollary}

\theoremstyle{definition}

\newcommand{\R}{\mathbb R}
\newcommand{\N}{\mathbb N}

\def\Aut{\mathrm{Aut}}

\numberwithin{equation}{section}

\title{Topological entropy for automorphisms of\\totally disconnected locally compact groups}

\author{Anna Giordano Bruno\footnote{Research supported by ``Progetti di Eccellenza 2011/12'' of Fondazione CARIPARO and by INdAM}}

\begin{document}

\maketitle

\begin{abstract}
We give a ``limit-free formula'' simplifying the computation of the topological entropy for topological automorphisms of totally disconnected locally compact groups. This result allows us to extend several basic properties of the topological entropy known to hold for compact groups.
\end{abstract}

\bigskip
\noindent \textrm{\small{Key words: topological entropy, totally disconnected locally compact group}}

\section{Introduction}

In this paper we consider the topological entropy for topological automorphisms of totally disconnected locally compact groups. 

For a locally compact group $G$ let $\mu$ denote a left Haar measure of $G$. Moreover, it is worth recalling immediately that a totally disconnected locally compact group $G$ has as a local base at $1$ the family $\mathcal B(G)$ of all open compact subgroups of $G$, as proved by van Dantzig in \cite{vD}.

\medskip
In \cite{AKM} Adler, Konheim and McAndrew introduced the topological entropy for continuous selfmaps of compact spaces, while later on Bowen in \cite{B} introduced it for uniformly continuous selfmaps of metric spaces, and this definition was extended to uniformly continuous selfmaps of uniform spaces by Hood in \cite{hood}. In particular, as explained in details in \cite{DG-islam}, for continuous endomorphisms of totally disconnected locally compact groups the topological entropy can be introduced as follows. 

Let $G$ be a totally disconnected locally compact group and $\phi:G\to G$ a continuous endomorphism. For every $U\in\mathcal B(G)$ and every positive integer $n$, the \emph{$n$-th $\phi$-cotrajectory} of $U$ is $$C_n(\phi,U)=U\cap \phi^{-1}(U)\cap\ldots\cap\phi^{-n+1}(U),$$ and the \emph{$\phi$-cotrajectory} of $U$ is
$$C(\phi,U)=\bigcap_{n=0}^\infty\phi^{-n}(U)=\bigcap_{n=1}^\infty C_n(\phi,U).$$
Note that $C_n(\phi,U)\in\mathcal B(G)$ for every positive integer $n$, while $C(\phi,U)$ is compact but not open in general, and it is the greatest $\phi$-invariant subgroup of $G$ contained in $U$. Moreover, for every positive integer $n$ the index $[U:C_n(\phi,U)]$ is finite as $U$ is compact and $C_n(\phi,U)$ is open.

The \emph{topological entropy of $\phi$ with respect to $U$} is given by the following limit, which is proved to exist in Lemma \ref{logalpha} below,
$$H_{top}(\phi,U)=\lim_{n\to \infty}\frac{\log[U:C_n(\phi,U)]}{n}.$$
The \emph{topological entropy} of $\phi$ is $$h_{top}(\phi)=\sup\{H_{top}(\phi,U):U\in\mathcal B(G)\}.$$

\medskip
In Section \ref{lf-sec} we prove the main result of this paper, that is the following limit-free formula for the computation of the topological entropy: if $G$ is a totally disconnected locally compact group, $\phi:G\to G$ is a topological automorphism and $U\in\mathcal B(G)$, then
\begin{equation}\label{lff}
H_{top}(\phi,U)=\log[\phi(C(\phi^{-1},U)):C(\phi^{-1},U)].
\end{equation}
This result extends to locally compact groups the limit-free formula proved in \cite{DG-limitfree} in the compact case. That formula was inspired by its counterpart for the algebraic entropy claimed by Yuzvinski in \cite{Y} and verified in \cite{DG-limitfree}. Indeed, in the abelian case the algebraic entropy and the topological entropy are strictly related by the Pontryagin duality (see \cite{DG-BT2,DG-BT,W}).

\medskip
Under some necessary assumptions, the limit-free formula in the compact case was given for continuous endomorphisms, so we leave the following open problem.

\begin{problem}
Is it possible to extend the limit-free formula in \eqref{lff} to continuous endomorphisms of locally compact groups? Under which assumptions?
\end{problem}

An interesting application of the limit-free formula in \eqref{lff} is given in \cite{BDG}, where it allows for a comparison of the topological entropy with the scale function $s(-)$ introduced by Willis in \cite{Willis,Willis2} for topological automorphisms $\phi:G\to G$ of totally disconnected locally compact groups $G$. Indeed, it follows from the so-called tidying procedure given in \cite{Willis2} that
$$s(\phi)=\min\{[\phi(C(\phi^{-1},U)):C(\phi^{-1},U)]:U\in\mathcal B(G)\};$$
so clearly $\log s(\phi)\leq h_{top}(\phi)$ in view of \eqref{lff}.
Furthermore, in \cite{BDG} a characterization is given of the topological automorphisms $\phi:G\to G$ of totally disconnected locally compact groups for which the equality $\log s(\phi)=h_{top}(\phi)$ holds.

\smallskip
In Section \ref{prop-sec}, as consequences of the limit-free formula in \eqref{lff}, we verify the basic properties of the topological entropy, namely, the so-called Invariance under conjugation, Monotonicity, Logarithmic law, Continuity with respect to inverse limits and weak Addition Theorem. Indeed, these properties are known to hold for continuous endomorphisms of compact groups (see \cite{AKM,St}), and now we extend them to the case of topological automorphisms of totally disconnected locally compact groups.

\medskip
In \cite{St} the so-called Addition Theorem was proved for continuous endomorphisms $\phi:G\to G$ of compact groups; indeed, if $N$ is a closed normal $\phi$-invariant subgroup of $G$ and $\overline \phi:G/N\to G/N$ is the continuous endomorphism induced by $\phi$, then $$h_{top}(\phi)=h_{top}(\phi\restriction_N)+h_{top}(\overline\phi).$$
We leave open the following problem to know whether the Addition Theorem holds in the setting of this paper; one should start studying the abelian case.

\begin{problem}
Let $G$ be a totally disconnected locally compact group, $\phi:G\to G$ a topological automorphism, $N$ a closed normal subgroup of $G$ such that $\phi(N)=N$, and $\overline\phi:G/N\to G/N$ the topological automorphism induced by $\phi$. Is it true that $h_{top}(\phi)=h_{top}(\phi\restriction_N)+h_{top}(\overline\phi)$? 
\end{problem}

\section{Limit-free computation of topological entropy}\label{lf-sec}

For a group $G$ and a subgroup $H$ of $G$, we denote by $G/H$ the set of all (left) cosets of $H$ in $G$, and we denote by $[G:H]$ the index of $H$ in $G$, that is the size of $G/H$. If $K$ is another subgroup of $G$, then we denote by $KH/H$ the family of all (left) cosets of $H$ in $G$ with representing elements in $K$, that is $KH/H=\{kH:k\in K\}$, and we denote by $[KH:H]$ the size of this family, generalizing the usual notion of index. If $H\subseteq K$ then we write simply $K/H$ and $[K:H]$, as usual.

We use several times the following easy-to-check properties. 

\begin{fact}\label{gi}
Let $G$ be a group and let $H,K,L$ be subgroups of $G$, with $H\subseteq K$. Then:
\begin{itemize}
\item[(a)] $[G:H]=[G:K][K:H]$;
\item[(b)] $[LH:H]=[L:H\cap L]$;
\item[(c)] $[K:H]\geq[K\cap L:H\cap L]$.
\end{itemize}
If $N$ is a normal subgroup of $G$, $q:G\to G/N$ is the canonical projection and $N\subseteq H$, then
\begin{itemize}
\item[(d)] $[K:H]=[q(K):q(H)]$.
\end{itemize}
\end{fact}

The following is the counterpart for the topological entropy of \cite[Lemma 1.1]{DGSZ} and \cite[Lemma 2.2]{DGS}; its proof follows the one of \cite[Lemma 3.1]{DG-limitfree}. 

\begin{lemma}\label{logalpha}
Let $G$ be a locally compact group, $\phi:G\to G$ a topological automorphism and $U\in\mathcal B(G)$. For every positive integer $n$ let $c_n:=[U:C_n(\phi,U)]$. Then:
\begin{itemize}
\item[(a)] $c_{n}$ divides $c_{n+1}$ for every $n>0$. 
\end{itemize}
For every $n>0$ let $\alpha_{n}=\frac{c_{n+1}}{c_{n}}=[C_{n}(\phi,U):C_{n+1}(\phi,U)]$. Then:
\begin{itemize}
\item[(b)] $\alpha_{n+1}\leq\alpha_{n}$ for every $n>0$;
\item[(c)] the sequence $\{\alpha_n\}_{n>0}$ stabilizes (i.e., there exist integers $n_0>0$ and $\alpha>0$ such that $\alpha_n=\alpha$ for every $n\geq n_0$);
\item[(d)] $H_{top}(\phi,U)=\log\alpha$.
\end{itemize}
\end{lemma}
\begin{proof}
Let $n>0$. Since there is no possibility of confusion, in this proof we denote $C_n(\phi,U)$ simply by $C_n$.

\smallskip
(a) Since $U\supseteq C_n\supseteq C_{n+1}$ it follows from Fact \ref{gi}(a) that $[U:C_{n+1}]=[U:C_n][C_n:C_{n+1}]$, and so $\frac{c_{n+1}}{c_n}=[C_{n}:C_{n+1}]$; in particular $c_{n}$ divides $c_{n+1}$.

\smallskip
(b) Since $C_{n+1}=C_n\cap \phi^{-1}(U)$, and using Fact \ref{gi}(b), we have $$[C_n:C_{n+1}]=[C_n:C_n\cap \phi^{-n}(U)]=[C_n\cdot\phi^{-n}(U):\phi^{-n}(U)].$$ Now $C_n\subseteq \phi^{-1}(C_{n-1})$, so $$[C_n\cdot\phi^{-n}(U):\phi^{-n}(U)]\leq[\phi^{-1}(C_{n-1})\cdot\phi^{-n}(U):\phi^{-n}(U)].$$
The map $G/\phi^{-n}(U)\to G/\phi^{-n+1}(U)$ induced by $\phi$ is injective, therefore the family of cosets
$(\phi^{-1}(C_{n-1})\cdot\phi^{-n}(U))/\phi^{-n}(U)$ has the same size of its image $(C_{n-1}\cdot\phi^{-n+1}(U))/\phi^{-n+1}(U)$, and so, applying also Fact \ref{gi}(b) and the fact that $C_n=C_{n-1}\cap \phi^{-n+1}(U)$, we have
\begin{equation*}\begin{split}
[\phi^{-1}(C_{n-1})\cdot\phi^{-n}(U):\phi^{-n}(U)]&=[C_{n-1}\cdot\phi^{-n+1}(U):\phi^{-n+1}(U)]=\\&=[C_{n-1}:C_{n-1}\cap \phi^{-n+1}(U)]=[C_{n-1}:{C_n}].
\end{split}\end{equation*}

\smallskip
(c) follows immediately from (b).

\smallskip
(d) By item (c) for $n_0>0$ we have $c_{n_0+n}=\alpha^n c_{n_0}$ for every $n\geq0$, and by the definition of topological entropy 
$$H_{top}(\phi,U)=\lim_{n\to\infty}\frac{\log c_n}{n}=\lim_{n\to \infty}\frac{\log(\alpha^n c_{n_0})}{n}=\log\alpha,$$
and this concludes the proof.
\end{proof}

Lemma \ref{logalpha} yields that the limit defining the topological entropy exists, and gives a precise description for the value of the topological entropy $H_{top}(\phi,U)$ of a topological automorphism $\phi:G\to G$ of a totally disconnected locally compact group $G$ with respect to $U\in\mathcal B(G)$; in particular, this value is in $\log\N_+$, so the value $h_{top}(\phi)\in\log\N_+\cup\{\infty\}$. 
%

\medskip
We are now in position to prove Theorem \ref{lf+}, that is, the main result of this paper, stated in \eqref{lff} in the Introduction. The following folklore fact is needed in the proof. 

\begin{lemma}\emph{\cite{DG-limitfree}}\label{intersect}
Let $G$ be a topological group and let $T$ be a closed subset of $G$. Then for every descending chain 
$B_1 \supseteq B_2 \supseteq \ldots \supseteq B_n \supseteq  \ldots$ of closed subsets of $G$ the intersection $B = \bigcap_{n=1}^\infty B_n $ is non-empty and 
 $\bigcap_{n=1}^\infty (B_{n} T)= B T$, whenever $B_1$ is countably compact. 
\end{lemma}

\begin{theorem}\label{lf+}
Let $G$ be a totally disconnected locally compact group, $\phi:G\to G$ a topological automorphism and $U\in\mathcal B(G)$. Then 
$$H_{top}(\phi,U)=\log[\phi(C(\phi^{-1},U)):C(\phi^{-1},U)].$$
\end{theorem}
\begin{proof}
By Lemma \ref{logalpha} there exist integers $n_0>0$ and $\alpha> 0$ such that $\alpha_n=\alpha$ for every $n\geq n_0$, where $\alpha_{n}=[C_{n}(\phi,U):C_{n+1}(\phi,U)]$ for every $n>0$, and $H_{top}(\phi,U)=\log\alpha$. So it suffices to prove that 
\begin{equation}\label{aim}
[\phi(C(\phi^{-1},U)):C(\phi^{-1},U)]=\alpha.
\end{equation}

Since $C(\phi^{-1},U)=U\cap \phi(C(\phi^{-1},U))$, and by Fact \ref{gi}(b), we have
\begin{equation}\label{eqa}\begin{split}
[\phi(C(\phi^{-1},U)):C(\phi^{-1},U)]&=[\phi(C(\phi^{-1},U)):U\cap \phi(C(\phi^{-1},U))]\\&=[\phi(C(\phi,U))\cdot U:U].
\end{split}\end{equation}
The family of cosets $(\phi(U)\cdot U)/U$ is finite, as $U$ is open and $\phi(U)\cdot U$ is compact, since both $U$ and $\phi(U)$ are compact. Consequently, the sequence $\{(\phi(C_n(\phi^{-1},U))\cdot U)/U\}_{n>0}$ is a descending chain of finite subsets of $(\phi(U)\cdot U)/U$ (note that $\phi(C_1(\phi^{-1},U))=\phi(U)$), hence it stabilizes, that is there exists $n_1>0$ such that 
\begin{equation}\label{n1n}
\phi(C_n(\phi^{-1},U))\cdot U=\phi(C_{n_1}(\phi^{-1},U))\cdot U\ \text{for every}\ n\geq n_1;
\end{equation}
in other words, 
\begin{equation}\label{iot}
\bigcap_{n=1}^\infty (\phi(C_{n}(\phi^{-1},U))\cdot U)=\phi(C_{n_1}(\phi^{-1},U))\cdot U.
\end{equation}
Now Lemma \ref{intersect} gives
\begin{equation}\label{lemmagives}
\bigcap_{n=1}^\infty (\phi(C_{n}(\phi^{-1},U))\cdot U)=\left(\bigcap_{n=1}^\infty \phi(C_n(\phi^{-1},U))\right)\cdot U=\phi(C(\phi^{-1},U))\cdot U,
\end{equation}
since $\bigcap_{n=1}^\infty \phi(C_n(\phi^{-1},U))=\phi(C(\phi^{-1},U))$. 
Therefore, \eqref{n1n}, \eqref{iot} and \eqref{lemmagives} yield
\begin{equation}\label{eqb}
\phi(C(\phi^{-1},U))\cdot U=\phi(C_{n_1}(\phi^{-1},U)) \cdot U=\phi(C_{n}(\phi^{-1},U)) \cdot U\ \text{for every}\ n\geq n_1.
\end{equation}
Let $n\geq\max\{n_0,n_1\}$. Then \eqref{eqa} and \eqref{eqb}, together with Fact \ref{gi}(b), give 
\begin{equation}\label{split}\begin{split}
[\phi(C(\phi^{-1},U)):C(\phi^{-1},U)]&=[\phi(C(\phi^{-1},U))\cdot U:U]\\&=[\phi(C_n(\phi^{-1},U))\cdot U:U]\\&=[\phi(C_n(\phi^{-1},U)):\phi(C_n(\phi^{-1},U))\cap U]\\&=[\phi(C_n(\phi^{-1},U)):C_{n+1}(\phi^{-1},U)].
\end{split}\end{equation}
Applying $\phi^{-n}$, we have the equalities 
$$\phi^{-n}(C_{n+1}(\phi^{-1},U))=C_{n+1}(\phi,U)\ \text{and}\ \phi^{-n}(\phi(C_n(\phi^{-1},U)))=C_n(\phi,U),$$ so, since $\phi^{-n}$ is an isomorphism, we have
\begin{equation}\label{lasteq}
[\phi(C_n(\phi^{-1},U)):C_{n+1}(\phi^{-1},U)]=[C_n(\phi,U):C_{n+1}(\phi,U)]=\alpha.
\end{equation}
By \eqref{split} and \eqref{lasteq} we can conclude that \eqref{aim} holds, so the thesis is proved.
\end{proof}

\section{The topological entropy of the inverse automorphism}

For a locally compact group $G$ and a left Haar measure $\mu$ on $G$, the \emph{modulus} is a group homomorphism $\Delta:\mathrm{Aut(G)}\to\R_+$ such that $\mu(\alpha E)=\Delta(\alpha)\mu(E)$ for every topological automorphism $\alpha\in \Aut(G)$ and every measurable subset $E$ of $G$ (see \cite{HR}). If $G$ is compact, then $\Delta(\phi)=1$ for every topological automorphism $\phi:G\to G$.

\begin{lemma}\label{Delta}
Let $G$ be a totally disconnected locally compact group, $\phi:G\to G$ a topological automorphism and $U\in\mathcal B(G)$. Then
\begin{itemize}
\item[(a)]  $\Delta(\phi)=\frac{\mu(\phi(U))}{\mu(U)}=\frac{[\phi(U):U\cap\phi(U)]}{[U:U\cap\phi(U)]}$.
\item[(b)] If $V\in\mathcal B(G)$ and $V\subseteq U\cap \phi(U)$, then $[\phi(U):V]=[U:V]\cdot\Delta(\phi)$.
\end{itemize}
\end{lemma}
\begin{proof}
(a) Follows from the fact that $\mu(\phi(U))=[\phi(U):U\cap \phi(U)]\cdot\mu(U\cap\phi(U))$ and $\mu(U)=[U:U\cap\phi(U)]\cdot \mu(U\cap\phi(U))$, since $U$ and $\phi(U)$ are compact and $U\cap\phi(U)$ is open, and so $[\phi(U):U\cap\phi(U)]$ and $[U:U\cap\phi(U)]$ are finite.

\smallskip
(b) The situation is described by the following diagram:
$$
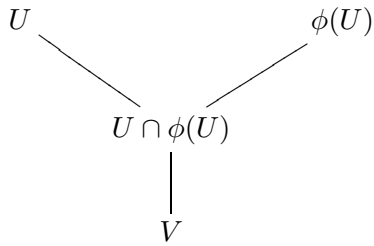
\xymatrix{
U \ar@{-}[dr]& & \phi(U)\ar@{-}[dl] \\
& U\cap\phi(U)\ar@{-}[d] \\
& V &
}$$
By Fact \ref{gi}(a) we have $$[U:V]=[U:U\cap\phi(U)][U\cap\phi(U):V]$$ and $$[\phi(U):V]=[\phi(U):U\cap\phi(U)][U\cap\phi(U):V].$$ Therefore, $$[\phi(U):V]=[U:V]\frac{[\phi(U):U\cap\phi(U)]}{[U:U\cap\phi(U)]}.$$
To conclude apply item (a).
\end{proof}

In the next proposition we give the local relation between the topological entropy of a topological automorphism of a totally disconnected locally compact group and the topological entropy of its inverse. They coincide precisely when the modulus of $\phi$ is $1$. So for example, as known, in the compact case a topological automorphism has the same topological entropy as its inverse.

\begin{proposition}\label{^-1}
Let $G$ be a totally disconnected locally compact group, $\phi:G\to G$ a topological automorphism and $U\in\mathcal B(G)$. Then
$$H_{top}(\phi^{-1},U)= H_{top}(\phi,U)-\log\Delta(\phi).$$
Hence, $$h_{top}(\phi^{-1})=h_{top}(\phi)-\log\Delta(\phi).$$
\end{proposition}
\begin{proof}
For every positive integer $n$ let $c_n=[U:C_n(\phi,U)]$ and $c_n^*=[U:C_n(\phi^{-1},U)]$.
According to items (a)--(c) of Lemma \ref{logalpha} applied to $\phi$ and $\phi^{-1}$ respectively, $H_{top}(\phi,U)=\log\alpha$ and $H_{top}(\phi^{-1},U)=\log \alpha^*$, where $\alpha$ is the value at which stabilizes the sequence $\alpha_n=\frac{c_{n+1}}{c_{n}}$ and $\alpha^*$ is the value at which stabilizes the sequence $\alpha^*_n=\frac{c_{n+1}^*}{c_{n}^*}$. For every $n>0$, since $\phi^n(C_n(\phi,U))=C_n(\phi^{-1},U)$, using that $\phi^n$ is an automorphism, and by Lemma \ref{Delta}(b), we have that
\begin{equation*}\begin{split}
c_n &=[U:C_n(\phi,U)]=[\phi^n(U):\phi^n(C_n(\phi,U))]=\\&=[\phi^n(U):C_n(\phi^{-1},U))=[U:C_n(\phi^{-1},U)]\cdot\Delta(\phi^n)=c_n^*\cdot \Delta(\phi^n).
\end{split}\end{equation*}
Therefore, since $\Delta$ is a homomorphism, $$\alpha=\frac{c_{n+1}}{c_n}=\frac{c_{n+1}^*}{c_n^*}\cdot\frac{\Delta(\phi^{n+1})}{\Delta(\phi^n)}=\alpha^*\Delta(\phi),$$ hence we have the equality $\log\alpha=\log\alpha^*+\log\Delta(\phi)$, which gives the first assertion of the proposition.

The second equality in the thesis follows from the first one taking the supremum for $U$ ranging in $\mathcal B(G)$.
\end{proof}

The following is an immediate consequence of Theorem \ref{lf+} and Proposition \ref{^-1}.

\begin{corollary}\label{lf-}
Let $G$ be a totally disconnected locally compact group, $\phi:G\to G$ a topological automorphism and $U\in\mathcal B(G)$. Then 
$$H_{top}(\phi,U)=\log[\phi^{-1}(C(\phi,U)):C(\phi,U)]+\log\Delta(\phi).$$
\end{corollary}
\begin{proof}
Theorem \ref{lf+} applied to $\phi^{-1}$ gives $$H_{top}(\phi^{-1},U)=\log[\phi(C(\phi^{-1},U)):C(\phi^{-1},U)].$$ Moreover, Proposition \ref{^-1} yields $H_{top}(\phi,U)= H_{top}(\phi^{-1},U)+\log\Delta(\phi)$, hence the thesis follows.
\end{proof}

\section{Properties of topological entropy}\label{prop-sec}

In this section we consider the basic properties of the topological entropy. When the group is compact it is known that they hold, but thanks to the limit-free formula proved in Theorem \ref{lf+} it is possible to extend them to the case of topological automorphisms of totally disconnected locally compact groups.

\medskip
We start proving in our setting that the function $H_{top}(\phi,-)$ is antimonotone. It is already known from \cite{DG-islam}.

\begin{lemma}\label{antim}
Let $G$ be a totally disconnected locally compact groupa and $\phi:G\to G$ a continuous endomorphism.
For any $U,V\in\mathcal B(K)$, if $U\subseteq V$, then $H_{top}(\phi,V)\leq H_{top}(\phi,U)$.
\end{lemma}
\begin{proof}
Since $U\subseteq V$, we have that $C_n(\phi,U)\subseteq C_n(\phi,V)$ for every $n>0$. Then, also by Fact \ref{gi}(a) we have for every $n>0$
$$[V:C_n(\phi,V)]\leq[V:C_n(\phi,U)]=[V:U][U:C_n(\phi,U)].$$
Applying the definition of topological entropy, that is, taking the logarithm, dividing by $n$ and letting $n\to\infty$, we find that $H_{top}(\phi,V)\leq H_{top}(\phi,U)$.
\end{proof}

By the definition of topological entropy we immediately derive from Lemma \ref{antim} that, in order to compute the ``global'' topological entropy, it suffices to take the supremum when $U$ ranges in a local base at $1$ of $K$:

\begin{corollary}\label{basesuff}
Let $G$ be a totally disconnected locally compact group, $\phi:G\to G$ a continuous endomorphism and $\mathcal B\subseteq \mathcal B(G)$ a local base at $1$. Then $h_{top}(\phi)=\sup\{H_{top}(\phi,U):U\in\mathcal B\}$.
\end{corollary}

We start proving that the topological entropy is invariant under conjugation.

\begin{proposition}\label{conj}
Let $G$ be a totally disconnected locally compact group and $\phi:G\to G$ a topological automorphism. Let $H$ be another totally disconnected locally compact group and $\xi:G\to H$ a topological isomorphism. Then $\xi(C(\phi,U))=C(\xi\phi\xi^{-1},\xi(U))$ for every $U\in\mathcal B(G)$, and so $h_{top}(\phi)=h_{top}(\xi\phi\xi^{-1})$.
\end{proposition}
\begin{proof}
Let $\psi=\xi\phi\xi^{-1}$. It is clear that, for $U\in\mathcal B(G)$,
\begin{equation}\label{uff}
\xi(C(\phi,U))=C(\psi,\xi(U)).
\end{equation}

Since $\xi:G\to H$ is a topological isomorphism, it induces a bijection $\mathcal B(G)\to \mathcal B(H)$, so by Theorem \ref{lf+} it suffices to check that 
\begin{equation}\label{suff}
[\phi(C(\phi^{-1},U)):C(\phi^{-1},U)]=[\psi(C(\psi^{-1},\xi(U))):C(\psi^{-1},\xi(U))].
\end{equation}
From \eqref{uff} applied to $\phi^{-1}$ and $\psi^{-1}$ we have $C(\psi^{-1},\xi(U))=\xi(C(\phi^{-1},U))$, so, using also the fact that $\xi$ is an isomorphism,
\begin{equation*}\begin{split}
[\psi(C(\psi^{-1},\xi(U))):C(\psi^{-1},\xi(U))]=[\psi(\xi(C(\phi^{-1},U))):\xi(C(\phi^{-1},U))]=\\
=[\xi(\phi(C(\phi^{-1},U))):\xi(C(\phi^{-1},U))]=[\phi(C(\phi^{-1},U)):C(\phi^{-1},U)];
\end{split}\end{equation*}
then \eqref{suff} is verified and this concludes the proof.
\end{proof}

In the next proposition we prove that the monotonicity of the topological entropy under taking subgroups and quotients.


\begin{proposition}\label{monotonicity}
Let $G$ be a totally disconnected locally compact group, $\phi:G\to G$ a topological automorphism and $N$ a closed normal subgroup of $G$ such that $\phi(N)=N$, and let $\overline\phi:G/N\to G/N$ be the topological automorphism induced by $\phi$. Then:
\begin{itemize}
\item[(a)] $C(\phi\restriction_N,U\cap N)=C(\phi,U)\cap N$ for every $U\in\mathcal B(G)$, and $h_{top}(\phi)\geq h_{top}(\phi\restriction_N)$;
\item[(b)] $C(\overline\phi,q(U))=q(C(\phi,U))$ for every $U\in\mathcal B(G)$ such that $N\subseteq U$, where $q:G\to G/N$ is the canonical projection; moreover, $h_{top}(\phi)\geq h_{top}(\overline\phi)$.
\end{itemize}
\end{proposition}
\begin{proof}
(a) The first assertion is clear. Let $V\in \mathcal B(N)$. Then $V= U\cap N$ for some $U\in\mathcal B(G)$. We verify that $$H_{top}(\phi\restriction_N,V)\leq H_{top}(\phi,U)$$ and this implies the wanted inequality.
By the first assertion in item (a), $C((\phi\restriction_N)^{-1},V)=C(\phi^{-1},U)\cap N$, so Fact \ref{gi}(c) gives
\begin{equation*}\begin{split}
[\phi\restriction_N(C((\phi\restriction_N)^{-1},V)):C((\phi\restriction_N)^{-1},V)]&=[\phi(C(\phi^{-1},U)\cap N):C(\phi^{-1},U)\cap N]\\&\leq[\phi(C(\phi^{-1},U)):C(\phi^{-1},U)].
\end{split}\end{equation*}
By Theorem \ref{lf+}, taking the logarithms, this inequality gives $H_{top}(\phi\restriction_N,V)\leq H_{top}(\phi,U)$.

\smallskip
(b) The proof of the first assertion is straightforward. Let $q:G\to G/N$ be the canonical projection and $q(U)\in\mathcal B(G/N)$, where $U\in\mathcal B(G)$; we can assume without loss of generality that $N\subseteq U$. Since $N\subseteq C(\phi^{-1},U)$, by the first assertion in item (b) and Fact \ref{gi}(d) we have 
\begin{equation*}\begin{split}
[\overline\phi(C(\overline\phi^{-1},q(U))):C(\overline\phi^{-1},q(U))]=[\overline\phi(q(C(\phi^{-1},U))):q(C(\phi^{-1},U))]=\\
=[q(\phi(C(\phi^{-1},U))):q(C(\phi^{-1},U))]]=[\phi(C(\phi^{-1},U)):C(\phi^{-1},U)].
\end{split}\end{equation*}
Theorem \ref{lf+} applies to this equality and gives $H_{top}(\overline\phi,q(U))=H_{top}(\phi,U)$, hence we can conclude that $h_{top}(\overline\phi)\leq h_{top}(\phi)$.
\end{proof}

We go on proving the so-called Logarithmic law for the topological entropy. Note that in the compact case $h_{top}(\phi)=h_{top}(\phi^{-1})$ and so the Logarithmic law extends to all integers. This is not possible in general in the locally compact case, as follows immediately from Proposition \ref{^-1}, where the modulus of $\phi$ is involved.

\begin{proposition}
Let $G$ be a totally disconnected locally compact group, $\phi:G\to G$ a topological automorphism and $k$ a positive integer. Then $h_{top}(\phi^k)=k\cdot h_{top}(\phi)$.
\end{proposition}
\begin{proof}
Let $U\in\mathcal B(G)$ and denote $C=C(\phi^{-1},U)$. We start noting that $$C=C(\phi^{-k},V),\ \text{where}\ V=C_k(\phi^{-1},U).$$
 Moreover, $\phi^{n+1}(C)\supseteq \phi^n(C)$ and $[\phi^{n+1}(C):\phi^n(C)]=[\phi(C):C]$ for every $n\geq0$, applying the automorphism $\phi^{-n}$. Then Fact \ref{gi}(a) gives 
\begin{align*}
[\phi^k(C):C]&=[\phi^k(C):\phi^{k-1}(C)]\cdot[\phi^{k-1}(C):\phi^{k-2}(C)]\cdot\ldots\cdot[\phi(C):C]\\
&=[\phi(C):C]^k.
\end{align*}
Therefore, $$[\phi^k(C(\phi^{-k},V)):C(\phi^{-k},V)]=[\phi^k(C):C]=[\phi(C):C]^k;$$
so Theorem \ref{lf+} implies $$H_{top}(\phi^k,V)=k\cdot H_{top}(\phi,U).$$
This immediately gives $h_{top}(\phi^k)\geq k\cdot h_{top}(\phi)$.
By Lemma \ref{antim}, since $V\subseteq U$, we have also $$H_{top}(\phi^k,U)\leq H_{top}(\phi^k,V)=k \cdot H_{top}(\phi,U),$$
so the missing converse inequality $h_{top}(\phi^k)\leq k\cdot h_{top}(\phi)$ holds as well, and this concludes the proof.
\end{proof}

We verify now the continuity of the topological entropy with respect to inverse limits.

\begin{proposition}
Let $G$ be a totally disconnected locally compact group and $\phi:G\to G$ a topological automorphism. If $\{N_i:i\in I\}$ is a directed system of closed normal subgroups of $G$ with $\phi(N_i)=N_i$ and $\bigcap_{i\in I} N_i=\{1\}$, then $G\cong \varprojlim G/N_i$ and $h_{top}(\phi)=\sup_{i\in I} h_{top}(\overline \phi_i)$, where $\overline \phi_i:G/N_i\to G/N_i$ is the continuous endomorphism induced by $\phi$.
\end{proposition}
\begin{proof}
By Proposition \ref{monotonicity}(b) the inequality $h_{top}(\phi)\geq\sup_{i\in I}h_{top}(\overline\phi_i)$ holds. So let $U\in\mathcal B(G)$. There exists $i\in I$ such that $N_i\subseteq U$. Let $q_i:G\to G/N_i$ be the canonical projection. Since $N_i\subseteq C(\phi^{-1},U)$, by Proposition \ref{monotonicity}(b) and Fact \ref{gi}(d) we have 
\begin{equation*}\begin{split}
[\overline\phi_i(C(\overline\phi_i^{-1},q_i(U))):C(\overline\phi_i^{-1},q_i(U))]=[\overline\phi_i(q_i(C(\phi^{-1},U))):q_i(C(\phi^{-1},U))]\\
=[q_i(\phi(C(\phi^{-1},U))):q_i(C(\phi^{-1},U))]=[\phi(C(\phi^{-1},U)):C(\phi^{-1},U)].
\end{split}\end{equation*}
Theorem \ref{lf+} applies to this equality and gives $H_{top}(\overline\phi_i,q_i(U))=H_{top}(\phi,U)$. Therefore $H_{top}(\phi,U)\leq h_{top}(\overline\phi_i)$, hence we can conclude that $h_{top}(\phi)\leq\sup_{i\in I}h_{top}(\overline\phi_i)$.
\end{proof}

The following property is the weak Addition Theorem for the topological entropy.

\begin{theorem}
Let $G$ and $H$ be totally disconnected locally compact groups, $\phi:G\to G$ and $\psi:H\to H$ topological automorphisms. Then $h_{top}(\phi\times \psi)=h_{top}(\phi)+h_{top}(\psi)$.
\end{theorem}
\begin{proof}
Let $\eta=\phi\times\psi$. Since the family $\mathcal B=\mathcal B(G)\times\mathcal B(H)$ is a local base at $1$ of $G\times H$, in view of Corollary \ref{basesuff} it suffices to take $U\in\mathcal B(G)$ and $V\in\mathcal B(H)$, and to verify that $$H_{top}(\eta,U\times V)=H_{top}(\phi,U)+H_{top}(\psi,V);$$
this holds since $C(\eta^{-1},U\times V)=C(\phi^{-1},U)\times C(\psi^{-1},V)$, so
$$[\eta(C(\eta^{-1},U\times V)):C(\eta^{-1},U\times V)]=[\phi(C(\phi^{-1},U)):C(\phi^{-1},U)]\cdot[\psi(C(\psi^{-1},V)):C(\psi^{-1},V)],$$
and Theorem \ref{lf+} concludes the proof.
\end{proof}

\end{document}